\documentclass[twoside]{amsart}
\usepackage{latexsym}
\usepackage{amssymb,amsmath,amsopn}
\usepackage{graphicx}   %To insert figures
\usepackage{color}     
\usepackage{a4wide}
\usepackage{hyperref} 
\newtheorem{thm}{Theorem}

\newtheorem{cor}{Corollary}

\theoremstyle{definition}

\newtheorem{ques}{Question}

\renewcommand{\Re}{\mathbb R}

\newcommand{\F}{\mathcal{F}}
\newcommand{\G}{\mathcal{G}}

\DeclareMathOperator{\conv}{conv}

\def\arXiv#1{arXiv:\href{http://arXiv.org/abs/#1}{#1}}

\def\MR#1{MR\href{http://www.ams.org/mathscinet-getitem?mr=#1}{#1}}

\title{Hexagon tilings of the plane that are not edge-to-edge}
\date{\today}
\author[D. Frettl\"{o}h]{Dirk Frettl\"{o}h}
\author[A. Glazyrin]{Alexey Glazyrin}
\author[Z. L\'{a}ngi]{Zsolt L\'{a}ngi}
\address{Technische Fakult\"{a}t, Universit\"{a}t Bielefeld, Postfach 100131, 33501 Bielefeld, Germany} 
\email{dfrettloeh@techfak.de}
\address{School of Mathematical \& Statistical Sciences, The University of Texas Rio Grande Valley, Brownsville, TX 78520, USA}
\email{alexey.glazyrin@utrgv.edu}
\address{MTA-BME Morphodynamics Research Group and Department of Geometry, Budapest University of Technology, Egry J\'{o}zsef utca 1., Budapest, Hungary, 1111 } 
\email{zlangi@math.bme.hu}


\begin{document} 

\begin{abstract}
An irregular vertex in a tiling by polygons is a vertex of one tile and belongs to the interior of an edge of another tile.
In this paper we show that for any integer $k\geq 3$, there exists a normal tiling of the Euclidean plane by convex hexagons of unit area with exactly $k$ irregular vertices. Using the same approach we show that there are normal edge-to-edge tilings of the plane by hexagons of unit area and exactly $k$ many $n$-gons ($n>6$) of unit area. A result of Akopyan yields an upper bound for $k$ depending on the maximal diameter and minimum area of the tiles. Our result complements this with a lower bound for the extremal case, thus showing that Akopyan's bound is asymptotically tight.
\end{abstract}

\maketitle

\section{Introduction}

Let $\F$ be a tiling of $\Re^2$ with convex polygons. Following \cite{DomLan}, we call a point $p \in \Re^2$ an \emph{irregular vertex} of $\F$ if it is a vertex of a tile, and a relative interior point of an edge of another one; compare Figure \ref{fig:final_tiling}: the edge labelled $E_3$ contains an irregular vertex. A tiling of $\Re^2$ having no irregular vertices is called \emph{edge-to-edge}. A tiling is \emph{normal} if for some positive $r$ and $R$, there is a disk of radius $r$ inside each tile and each tile is contained in a disk of radius $R$ \cite{GS, Schulte}. A tiling is a \emph{unit area tiling} if the area of every tile is one, and it has \emph{bounded perimeter} if perimeters of all tiles are bounded from above and from below by positive numbers. Clearly, a unit area tiling with bounded perimeter is normal, and a tiling with finitely many prototiles is both normal and has bounded perimeter.

In \cite{Nandakumar} Nandakumar asked whether there is a tiling of the plane by pairwise non-congruent triangles of equal area and equal perimeter. Kupavskii, Pach and Tardos proved that there is no such tiling \cite{KPT18_2}. 
Nandakumar also asked \cite{Nandakumar} if there is a tiling of the plane with pairwise non-congruent triangles with equal area and bounded perimeters. This was answered affirmatively in \cite{KPT18_1} and in \cite{Frettloh}. The tilings constructed in these papers have many irregular vertices. 
Later in \cite{F-R} it was shown that there is such a tiling without any irregular vertices. In the same paper it was shown that there is a tiling of the plane with pairwise non-congruent hexagons with equal area and bounded perimeters such that the tiling \emph{has} irregular vertices.
It is worth noting that while in general the requirement of having no irregular vertices is more restrictive than the requirement of having some, for tilings with convex hexagons as tiles it is the other way around: it is difficult to find hexagon tilings with many irregular vertices.
This observation is illustrated by the fact that the hexagonal tiling in \cite{F-R} has only two irregular vertices. 
Motivated by these results, the main goal of this paper is to investigate the number $k$ of irregular vertices in a unit area and bounded perimeter tiling of the Euclidean plane by convex hexagons. Our main result is the following.

\begin{thm}\label{thm:main}
In any normal tiling of the Euclidean plane with convex hexagons there are finitely many irregular vertices. Furthermore, for any integer $k \geq 3$, there is a unit area tiling of the plane with bounded perimeter convex hexagons such that the number of irregular vertices of the tiling is equal to $k$.
\end{thm}

In Section~\ref{sec:proof} we prove Theorem~\ref{thm:main}. Our proof is based on results of Stehling \cite{ste88} and Akopyan \cite{ako18}, which provide an upper bound on the number of irregular vertices in a hexagon tiling (cf. Theorems~\ref{thm:finite} and \ref{thm:indexupperbound}, respectively). In Section \ref{sec:asymptotic} we show that the  upper bound of Akopyan in Theorem~\ref{thm:nonhexlowerbound} is asymptotically tight, using a simplified version of the construction in Section~\ref{sec:proof}. We summarize our corresponding results in Theorems~\ref{thm:nonhexlowerbound} and \ref{thm:asymptotic}. Section~\ref{sec:ques} lists some additional questions.

In the sequel we denote the Euclidean plane by $\Re^2$. Furthermore, for any points $p,q \in \Re^2$ we denote the closed segment connecting $p$ and $q$ by $[p,q]$, the convex hull of a set $S \subset \Re^2$ by $\conv S$, and the Euclidean norm of a point $p \in \Re^2$ by $| p |$.

\section{Proof of Theorem~\ref{thm:main}} \label{sec:proof}

In his thesis \cite{ste88}, Stehling proved the following result.

\begin{thm}[Stehling]\label{thm:finite}
In a normal tiling of a plane with convex polygons such that all of them have at least six edges, all but a finite number of tiles are hexagons.
\end{thm}

Akopyan \cite{ako18} found a short proof of this theorem with concrete bounds on the number of non-hexagonal tiles. To formulate his result, we define the \emph{index} of a polygon as the number of its edges minus $6$.

\begin{thm}[Akopyan]\label{thm:indexupperbound}
If the plane is tiled with convex polygons, each having at least six edges, area at least $A$ and diameter not greater than $D$, then the sum of the indices of the tiles is not greater than $\frac{2\pi D^2}{A} - 6$.
\end{thm}

This theorem immediately implies the finiteness of the number of irregular vertices in a normal tiling by hexagons.

\begin{cor}\label{cor:upperbound}
In a normal tiling by hexagons, there is a finite number of irregular vertices.
\end{cor}

\begin{proof}
For each hexagon with an irregular vertex in the interior of one of its edges, we can subdivide this edge into several segments by all irregular vertices it contains and obtain a polygon with a positive index. By Theorem \ref{thm:indexupperbound}, the number of polygons of positive index is finite and, therefore, the number of irregular vertices in the initial tiling is finite, too.
\end{proof}

Now we prove the second part of Theorem~\ref{thm:main}.
%We prove a counterpart of Corollary \ref{cor:upperbound} under the stronger requirements of unit area and bounded perimeter imposed on a tiling.

\begin{proof}
Consider a unit area regular hexagonal tiling $\F$. Let $H$ be a tile of $\F$, and let $E$ be an edge of $H$. Without loss of generality, we may assume
that the centre of $H$ is the origin $o$. Let $R$ denote the angular region with the apex $o$ and with the boundary consisting of the two half-lines starting at $o$ and passing through the endpoints of $E$ (cf. Figure~\ref{fig:hexagonal_lattice}). Let $T$ be the regular triangle $H \cap R$ of area $1/6$.

\begin{figure}[ht]
\begin{center}
\includegraphics[width=.5\textwidth]{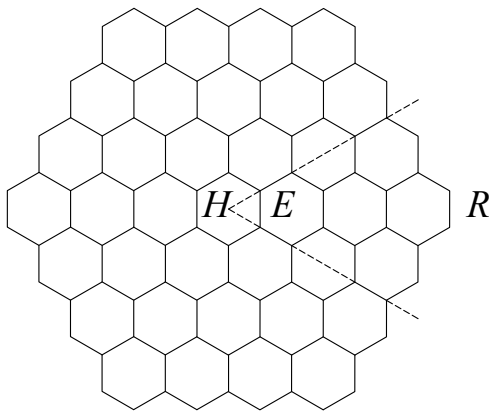}
\caption{The regular hexagonal lattice tiling $\F$}
\label{fig:hexagonal_lattice}
\end{center}
\end{figure}

Let $\alpha$ be an arbitrary angle measure from $(0,\frac{\pi}{3}]$ and let $f : \Re^2 \to \Re^2$ be an area-preserving linear transformation that maps $T$ to an isosceles triangle $T'$ with the apex $o$ and angle $\alpha$ at $o$. Then $R'=f(R)$ is an angular region with angle $\alpha$ at its apex. We denote the two half lines bounding $R'$ by $L_1$ and $L_2$. Applying $f$ also to the restriction of $\F$ to $R$ we obtain a tiling $\F'$ of $R'$.
Let the vertices of $\F'$ on $L_1$ be denoted by $p_1, p_2, \ldots$, in the consecutive order, and the vertices on $L_2$ be denoted by $q_1, q_2, \ldots$, where $q_i$ are symmetric to $p_i$ with respect to the symmetry axis of $R'$.
Then the tiles in $\F'$ are the following:
\begin{itemize}
\item the triangle $T'$ of area $\frac{1}{6}$;
\item affinely regular hexagons of unit area;
\item trapezoids of area $\frac{1}{2}$, obtained by dissecting an affinely regular hexagon into two congruent parts by a diagonal connecting two opposite vertices.
\end{itemize}
Note that $\F'$ is symmetric with respect to the symmetry axis of $R'$, the edge of each trapezoid in $\F'$ is either $[p_{2m},p_{2m+1}]$ or $[q_{2m},q_{2m+1}]$ for some positive integer $m$ (cf. Figure~\ref{fig:affine_copy}), and the two angles of the trapezoids on these edges are acute.

\begin{figure}[ht]
\begin{center}
\includegraphics[width=.6\textwidth]{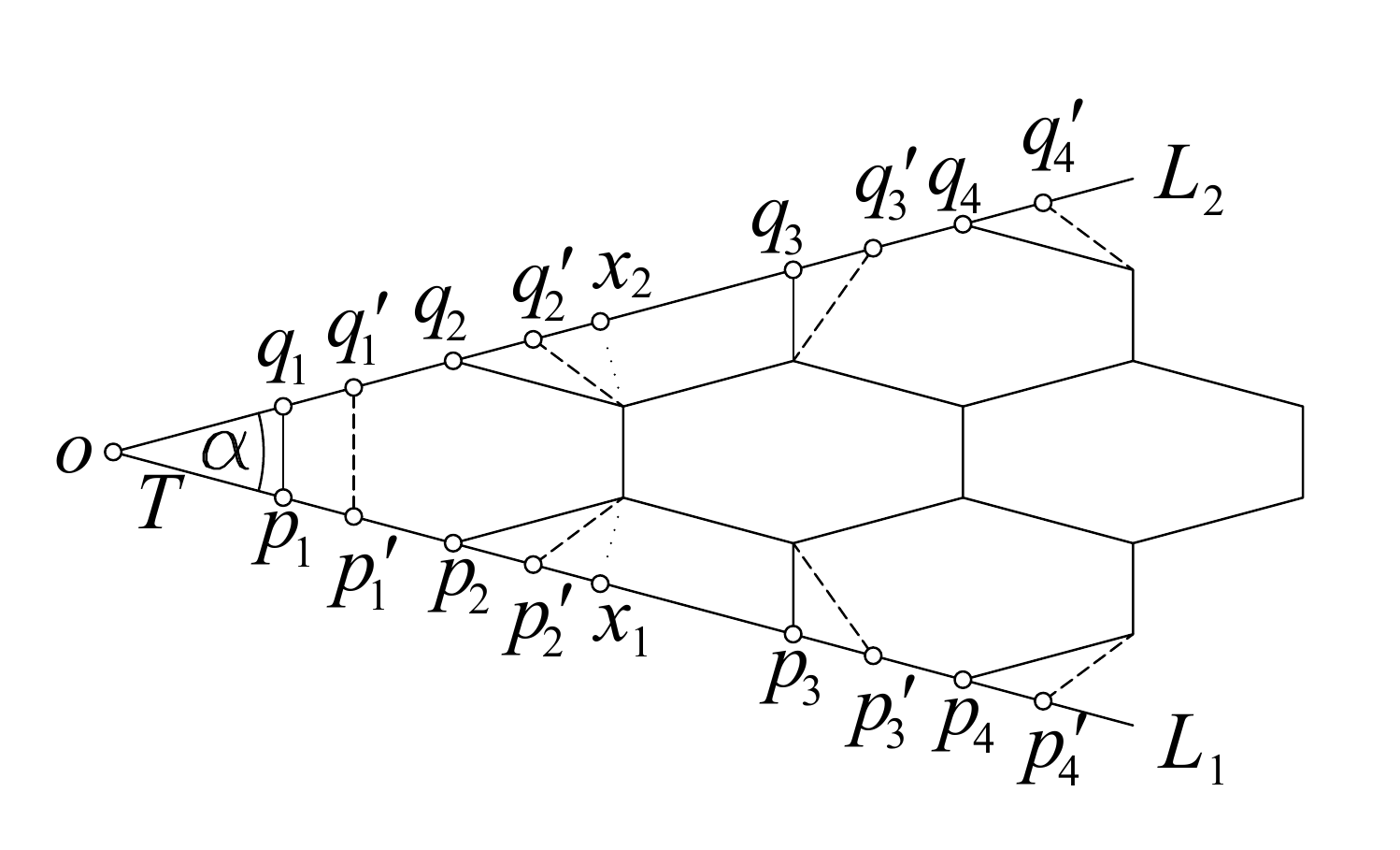}
\caption{Modifying an affine copy $\F'$ of $\F \cap R$}
\label{fig:affine_copy}
\end{center}
\end{figure}

In the next part we replace all points $p_i$ and $q_i$ by points $p_i'$ and $q_i'$, respectively, to obtain a tiling $\G$ of $R'$, combinatorially equivalent to $\F'$, which has the following properties:
\begin{itemize}
\item $\G$ is axially symmetric to the symmetry axis of $R'$,
\item the triangle $T''$ corresponding to $T'$ is an isosceles triangle of area $\frac{1}{3}$, with $o$ as its apex;
\item the hexagons corresponding to the hexagons in $\F'$ have unit area;
\item the quadrangles corresponding to the trapezoids in $\F'$ are also trapezoids of area $\frac{1}{2}$, and the two angles on the edges $[p'_{2m},p'_{2m+1}]$ or $[q'_{2m},q'_{2m+1}]$ ($m \in \mathbb{N}$) are acute.
\end{itemize}

Let $H'$ be the hexagon cell of $\F'$ with $p_1, p_2, q_1, q_2 \in H'$. Let the remaining two vertices of $H'$ be $z_1$ and $z_2$ such that $z_1$ is closer to $L_1$. For $i=1,2$, let $x_i$ be the orthogonal projection of $z_i$ onto $L_i$. It is an elementary computation to show that if $\alpha = \frac{\pi}{3}$, then the area $A$ of the convex pentagon $\conv \{ o, x_1, z_1, z_2, x_2\}$ is $\frac{4}{3}$, and if $0 < \alpha < \frac{\pi}{3}$, then $A > \frac{4}{3}$.
Now we apply modifications as follows. Replace the segment $[p_1,q_1]$ by the parallel segment $[p'_1,q'_1]$ such that $p'_1 \in L_1$, $q'_1 \in L_2$, and the isosceles triangle $T'' = \conv \{ o, p'_1, q'_1 \}$ has area $\frac{1}{3}$. Then the area of the hexagon $\conv \{ p'_1, x_1, z_1, z_2, x_2, p'_2\}$ is greater than $1$. Thus, there are points $p'_2 \in [p'_1,x_1]$, $q'_2 \in [q'_1,x_2]$ such that the area of $\conv \{ p'_1,p'_2, z_1, z_2, q'_2, q'_1\}$ is one, and this hexagon is axially symmetric to the symmetry axis of $R'$; i.e. $|p'_2-p'_1| = |q'_2-q'_1|$. Until now, we have modified the construction to satisfy the conditions in the previous list apart from the fact that the quadrangles with $[p'_2,p_3]$ and $[q'_2,q_3]$ as edges have area $\frac{5}{12}$.

Observe that if $[a,b]$ and $[c,d]$ are parallel edges of a trapezoid, then moving $[c,d]$ parallel to $[a,b]$ does not change the area of the trapezoid.
Using this observation, we set $v_1=p'_2-p_2$, $v_2 = q'_2-q_2$, and $p'_i = p_i+v_1$ and $q'_i = q_i+v_2$ for all integers $i \geq 2$.
Then, in the resulting tiling $\G$ any quadrangle having $[p'_{2m},p'_{2m+1}]$ or $[q'_{2m},q'_{2m+1}]$ ($m \in \mathbb{N}$) as an edge has area $\frac{1}{2}$, and the two angles lying on this edge are acute. Using the same observation, we may derive the analogous statement for the hexagons containing $[p'_{2m-1},p'_{2m}]$ or $[q'_{2m-1},q'_{2m}]$ ($m \in \mathbb{N}$, $m \geq 2$) as an edge. Clearly, the tiling $\G$ obtained in this way satisfies all the required conditions.

Observe that if we rotate $\G$ around $o$ by angle $\alpha$, and glue the pairs of quadrangles meeting in an edge on $L_1$ or $L_2$, we obtain a tiling of an angular region of angle $2\alpha$, in which any tile contained entirely in the interior of the region is a hexagon of unit area.

\begin{figure}[ht]
\begin{center}
\includegraphics[width=.6\textwidth]{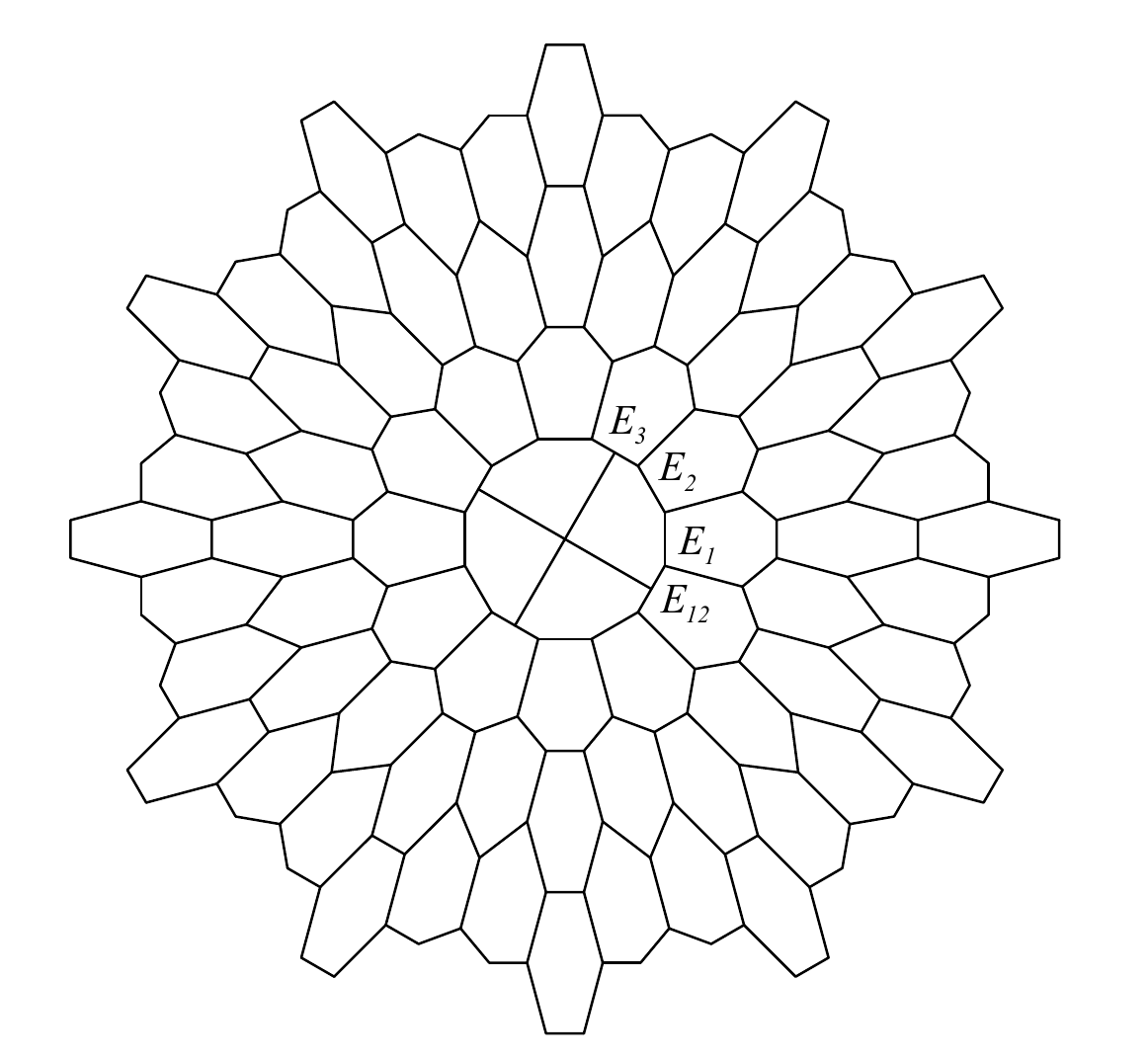}
\caption{The tiling $\G_4$ constructed in the proof of Theorem~\ref{thm:main}}
\label{fig:final_tiling}
\end{center}
\end{figure}

Now we prove Theorem~\ref{thm:main}.
Let $k \geq 3$ be fixed, and let $\G$ be the tiling constructed in the previous part of the proof with angle $\alpha = \frac{2\pi}{3k}$.
Let us rotate $\G$ around $o$ in counterclockwise direction and glue the pairs of quadrangles meeting on $L_1$ or $L_2$. Continuing this process we obtain a tiling of $\Re^2$, in which all tiles are unit area hexagons, outside a regular $(3k)$-gon $P$ centred at $o$, with area $k$. This regular $(3k)$-gon is obtained as the union of all the rotated copies of the triangle $T''$. Let $E_i$, $i=1,2,\ldots, 3k$ denote the edges of $P$. Then, dissecting $P$ with $k$ segments, connecting $o$ to the midpoint of an edge $E_{3j}$, $j=1,2,\ldots,k$, yields a tiling of $P$ with $k$ unit area hexagons (cf. Figure~\ref{fig:final_tiling}). The $k$ midpoints, used as vertices of the tiles in $P$ are not vertices of any other tiles, and thus they are irregular. On the other hand, all other vertices of the obtained tiling, which we denote by $\G_k$, are clearly regular. Finally, $\G_k$ consists of five prototiles, and hence, it has bounded perimeter and is normal.
\end{proof}

\section{Tightness of the upper bound} \label{sec:asymptotic}
Using the construction from Theorem \ref{thm:main} and subdividing the edges of polygons by irregular vertices, as in Corollary \ref{cor:upperbound}, we can find the counterpart of Stehling's theorem, i.e. we show the existence of normal tilings in which all tiles have at least six edges and the number of non-hexagon tiles is an arbitrarily large finite number. However, a simplified version of the construction from Theorem \ref{thm:main} provides also unit area and bounded perimeter tilings with the same property.

\begin{thm}\label{thm:nonhexlowerbound}
For any integer $k \geq 2$, there is a unit area tiling of the plane by convex polygons with at least six edges such that the set of the perimeters of the tiles is bounded, and the number of non-hexagon tiles is equal to $k$.
\end{thm}

\begin{proof}

We follow the proof of Theorem \ref{thm:main} and obtain a tiling $\F'$ of the angular region with the angle measure $\alpha=\frac {\pi} {3k}$. Instead of modifying this tiling we just glue $6k$ copies of $\F'$ to obtain a tiling of the plane (cf. Figure~\ref{fig:nonhex_tiling}).

\begin{figure}[ht]
\begin{center}
\includegraphics[width=\textwidth]{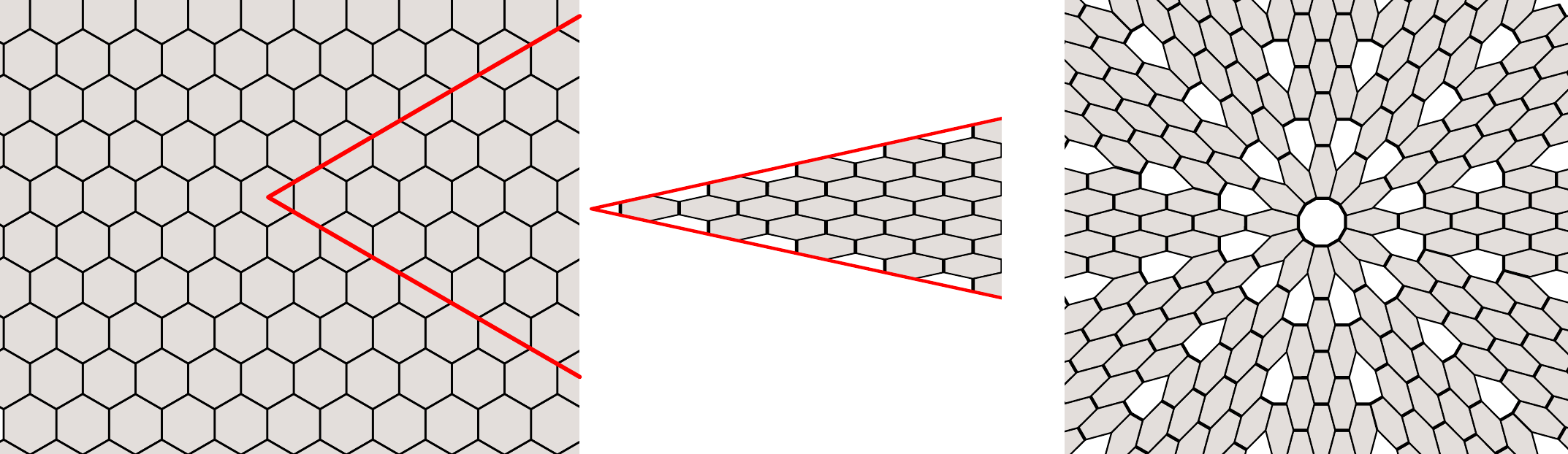}
\caption{The tiling used in the proofs of Theorems~\ref{thm:nonhexlowerbound} and
\ref{thm:asymptotic} for $k=2$}
\label{fig:nonhex_tiling}
\end{center}
\end{figure}

This is a tiling of the plane by hexagons of unit area with two prototiles, and a regular $(6k)$-gon of area $k$ in the centre. For $k=2$, we connect the opposite vertices of the central $12$-gon to divide it into two heptagons of unit area. For $k\geq 3$, we connect the centre of the $(6k)$-gon with every sixth vertex to subdivide it into $k$ congruent octagons. Overall, we obtain an edge-to-edge tiling $\F_k$ by unit area polygons with just three prototiles.
\end{proof}

We use the construction above to show that the bound in Theorem \ref{thm:indexupperbound} is asymptotically tight.
To state our result, for any normal tiling $\F$ of the plane by convex polygons having at least six edges, we denote the sum of the indices of the tiles in $\F$ by $i(\F)$, and call it the \emph{total index} of $\F$.

\begin{thm} \label{thm:asymptotic}
There is a sequence of normal tilings $\F_k$ by convex polygons having at least six edges such that 
\[ \lim_{k \to \infty} \frac{\frac{2 \pi D_k^2}{A_k} -6}{i(\F_k)} = \frac{4}{3}, \]
where, analogously to Theorem \ref{thm:indexupperbound}, $D_k$ denotes the maximal diameter, and 
$A_k$ denotes the minimal area of the tiles in $\F_k$, respectively.
\end{thm}

\begin{proof}
We use the construction from the proof of Theorem~\ref{thm:nonhexlowerbound} without 
partitioning the central $(6k)$-gon into octagons. Let us denote the resulting tiling 
by $\F_k$. All tiles in $\F_k$ but one are unit area hexagons, of two distinct 
shapes. The central tile in $\F_k$ is a $(6k)$-gon, hence the
sum of indices in $\F_k$ is $i(\F_k) = 6k-6$. Let us compare this value 
with the bound $\frac{2 \pi D_k^2}{A_k}-6$ from Theorem \ref{thm:indexupperbound}.

Obviously the smallest area $A_k$ of the tiles in $\F_k$ is one.
We show that all three prototiles have the same diameter $D_k= \frac{2}{\sqrt{3 \sin(\pi/3k)}}$. 

Let us call the area preserving linear transformation that 
maps the cone with angle $\frac{\pi}{3}$ to the cone with angle 
$\alpha$ by $f$. The image of a regular hexagon with diameter $2r_1$ under 
$f$ is an affinely regular hexagon with diameter $2r_k$, see Figure \ref{fig:affine-6gons}. 

\begin{figure}
\[ \includegraphics[width=60mm]{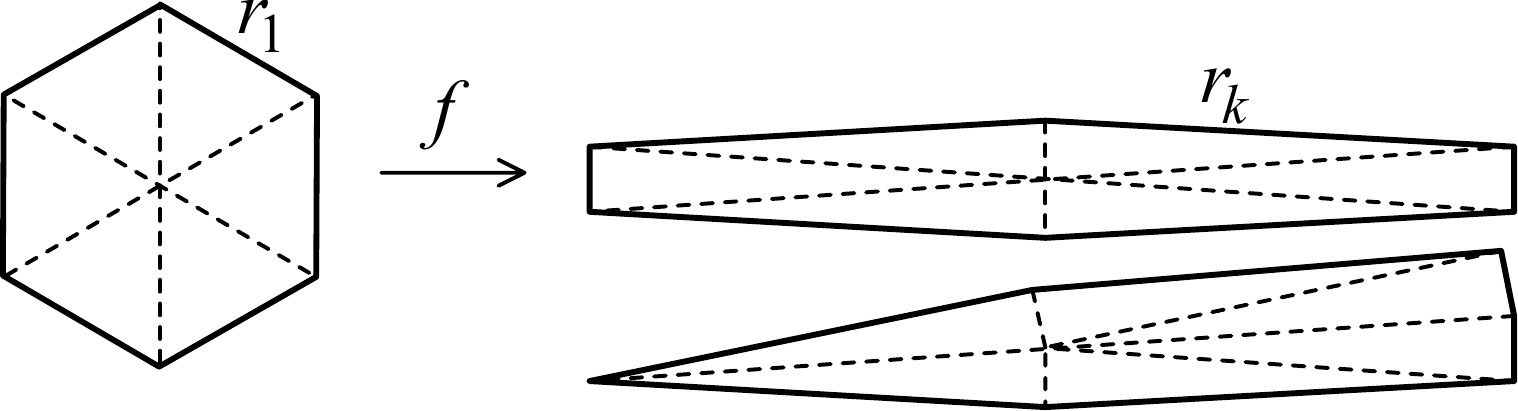} \]
\caption{Diameter of deformed hexagons is $2r_k$
\label{fig:affine-6gons}}
\end{figure}

Note that $f$ distorts the basic regular triangles in the regular hexagons of edge 
length $r_1$ into isosceles triangles with longer edge lengths $r_k$.
Furthermore, as the central $(6k)$-gon can be dissected into $6k$ isosceles triangles intersecting in $o$ with longest 
edge lengths equal to $r_k$, the diameter of this $(6k)$-gon is also $2r_k$.
Since area is preserved by $f$, we get
$\frac{1}{2}r_k^2 \sin(\frac{\pi}{3k})=\frac{1}{6}$. This yields 
\[ r_k = \frac{1}{\sqrt{3 \sin(\pi/3k)}}, \; \mbox{ which implies } D_k = 2 r_k = 
\frac{2}{\sqrt{3 \sin(\pi/3k)}}. \]
Let us plug this into the formula for the upper bound on $i(\F)$ in Theorem \ref{thm:indexupperbound}, and compute the 
ratio of the obtained expression with the total index $i(\F_k) = 6k-6$ of our construction. Then we have
\[ 
\lim_{k \to \infty} \frac{1}{6k-6} \Big( \frac{2 \pi D_k^2}{A_k} -6 \Big) =
\lim_{k \to \infty} \frac{2 \pi \frac{4}{3 \sin\left(\frac{\pi}{3k} \right)}-6}{6k-6}
=  \lim_{k \to \infty} \frac{4 \frac{\frac{\pi}{3k}}{\sin\left(\frac{\pi}{3k}\right)} - \frac{3}{k}} {3-\frac{3}{k}} = \frac{4}{3},   \]
where in the last step we used the well-known limit $\lim\limits_{x \to 0} \frac{\sin (x)}{x} = 1$.
\end{proof}                    

\section{Open questions}\label{sec:ques}

Our first question is motivated by the result in \cite{KPT18_2}.

\begin{ques}
Let $k \geq 1$ be an arbitrary integer. Prove or disprove that there is a unit area tiling of $\Re^2$ with equal perimeter convex hexagons such that the number of irregular vertices of the tiling is at least $k$. Similarly, are there equal area and equal perimeter tilings of $\Re^2$ in which every tile has at least six edges, and at least $k$ tiles have more?
\end{ques}

\begin{ques}
Let $k \geq 1$ be an arbitrary integer. What is the smallest value of $n$ such that for any $k \geq 3$, there is a tiling $\F$ of $\Re^2$ with convex (or possibly non-convex) hexagons, and with $n$ prototiles, such that the number of irregular vertices in $\F$ is at least $k$? Our proof of Theorem \ref{thm:main} shows $n \le 5$. In particular we ask whether there are monohedral ($n=1$) tilings satisfying this property?
\end{ques}

\begin{ques}
Can we improve the ratio $\frac{4}{3}$ in Theorem \ref{thm:asymptotic}?
That is, are there sequences of normal tilings by convex polygons with at least six edges for which the ratio in Theorem \ref{thm:asymptotic}
tends to less than $\frac{4}{3}$? Are there such sequences of unit area tilings with bounded perimeters?
\end{ques}

\begin{ques}
It is also tempting to ask about an analogue of Theorem \ref{thm:main} in higher dimensions.
The question becomes much harder, because we do not even know what the equivalents of the
hexagons in dimension 3 are. In other words: let $n$ be the minimal number of faces of a tile 
in a normal tiling by convex tiles in $\Re^3$. What is the maximal value of $n$? Engel showed 
\cite{engel81} that it is at least 38 by constructing a convex polyhedron with 38 faces that 
tiles $\Re^3$ in a face-to-face manner.
\end{ques}

\section*{Acknowledgements}
We thank the BIRS and the organizers of the meeting ``Soft Packings, Nested Clusters, and 
Condensed Matter'' at Casa Matem\'atica Oaxaca where we obtained the results in this paper. 
The first author is supported by the Research Centre for Mathematical Modelling (RCM$^2$) of
Bielefeld University. 
The third author is partially supported by the NKFIH Hungarian Research Fund grant 119245, the J\'anos Bolyai Research Scholarship of the Hungarian Academy of Sciences, and grants BME FIKP-V\'IZ and \'UNKP-19-4 New National Excellence Program by the Ministry of Innovation and Technology.


\begin{thebibliography}{10}

\bibitem{ako18}
A.~Akopyan.
\newblock {\em On the number of non-hexagons in a planar tiling.}
\newblock  C.~R.~Math.~Acad.~Sci.~Paris {\bf 356} (2018), no. 4, 412--414. 
%\arXiv{1805.01652}
%\MR{3787530}

\bibitem{DomLan}
G.~Domokos, Z.~L\'angi.
\newblock {\em On some average properties of convex mosaics.}
\newblock  Experiment. Math. (2019), DOI: 10.1080/10586458.2019.1691090.

\bibitem{engel81}
P.~Engel. 
\newblock \emph{Ueber Wirkungsbereichsteilungen yon kubischer Symmetrie.}
\newblock Z.~Kristallogr. \textbf{154} (1981) 199--215.

\bibitem{Frettloh}
D.~Frettl\"oh.
\newblock{ \em Noncongruent equidissections of the plane.}
\newblock Discrete geometry and symmetry, 171--180, Springer Proc. Math. Stat. \textbf{234}, Springer, Cham, 2018. 

\bibitem{F-R}
D.~Frettl\"oh, C.~Richter.
\newblock{ \em Incongruent equipartitons of the plane.}
\arXiv{1905.08144}

\bibitem{GS}
B.~Gr\"unbaum, G.C.~Shephard.
\emph{Tilings and Patterns}.
W.H.\ Freeman, New York (1987).

\bibitem{KPT18_1}
A. Kupavskii, J. Pach and G. Tardos.
\newblock {\em Tilings of the plane with unit area triangles of bounded diameter.}
\newblock Acta Math. Hungar. \textbf{155} (2018), no. 1, 175--183.

\bibitem{KPT18_2}
A. Kupavskii, J. Pach and G. Tardos.
\newblock {\em Tilings with noncongruent triangles.}
\newblock European J. Combin. \textbf{73} (2018), 72--80.

\bibitem{Nandakumar}
R.\ Nandakumar: {\tt https://nandacumar.blogspot.com/2014/12/filling-plane-with-non-congruent-pieces.html}

\bibitem{Schulte}
E. Schulte.
\newblock {\em Tilings.}
\newblock In: P. M. Gruber and J. M. Wills (eds): Handbook of convex geometry, {V}ol. {A}, {B}, 899--932, North-Holland, Amsterdam, 1993.

\bibitem{ste88}
T.~Stehling.
\newblock {\em \"{U}ber kombinatorische und graphentheoretische {E}igenschaften normaler 
{P}flasterungen.}
\newblock ProQuest LLC, Ann Arbor, MI, USA, 
\newblock PhD Thesis (Dr.rer.nat), Universit\"{a}t Dortmund, Germany, 1989.
\MR{2714383}

\end{thebibliography}
\end{document}